\documentclass{amsart}

\newtheorem{theorem}{Theorem}[section]
\newtheorem{lemma}[theorem]{Lemma}

\newtheorem{proposition}[theorem]{Proposition}
\newtheorem{corollary}[theorem]{Corollary}

\theoremstyle{definition}

\newtheorem{conjecture}[theorem]{Conjecture}

\theoremstyle{remark}

\numberwithin{equation}{section}

\usepackage[pdftex,bookmarks=true]{hyperref}
\usepackage{amscd}
      \usepackage[nohug,heads=littlevee]{diagrams}
      \diagramstyle[labelstyle=\scriptstyle]

\newcommand{\isomorphic}{\simeq}

\begin{document}

\title{Generic Representation Theory of the Heisenberg Group}

\author{Michael Crumley}
\address{Department of Mathematics, The University of Toledo,
Toledo, Ohio 43606}
\email{mikecrumley@hotmail.com}


\date{October 2010.}



\begin{abstract}
In this paper we extend a result for representations of the
Additive group $G_a$ given in \cite{SFB} to the Heisenberg group
$H_1$. Namely, if $p$ is greater than $2d$ then all
$d$-dimensional characteristic $p$ representations for $H_1$ can
be factored into commuting products of representations, with each
factor arising from a representation of the Lie algebra of $H_1$,
one for each of the representation's Frobenius layers. In this
sense, for a fixed dimension and large enough $p$, all
representations for $H_1$ look generically like representations
for direct powers of it over a field of characteristic zero.

The reader may consult chapter 13 of \cite{MyDissertation} for a
fuller account of what follows.
\end{abstract}

\maketitle


\section{Introduction}
Denote by $G_a$ and $H_1$ the Additive and Heisenberg groups
respectively over a fixed field $k$, i.e.~the space of all
unipotent upper triangular matrices of the form
\[\left(%
\begin{array}{cc}
  1 & x \\
  0 & 1 \\
\end{array}%
\right) \hspace{1cm} \text{and} \hspace{1cm}
\left(%
\begin{array}{ccc}
  1 & x & z \\
  0 & 1 & y \\
  0 & 0 & 1 \\
\end{array}%
\right)
\]
Throughout we prefer to think of these as affine group schemes
(see \cite{waterhouse}), i.e.~as representable functors on
$k$-algebras represented by the Hopf algebras (see
\cite{hopfalgebras})
\begin{gather*}
A = k[x] \\
\Delta:x \mapsto x \otimes 1 + 1 \otimes x \\
 \varepsilon: x \mapsto 0
\end{gather*}
and
\begin{gather*}
A = k[x,y,z] \\
 \Delta: x \mapsto 1 \otimes x+ x \otimes 1, \hspace{.5cm} y
\mapsto 1 \otimes y + y \otimes 1, \hspace{.5cm}
 z \mapsto 1 \otimes z + x \otimes y + z \otimes 1 \\
 \varepsilon: x,y,z \mapsto 0
\end{gather*}
respectively.  The following is well known (see theorems 12.2.1
and 13.3.1 of \cite{MyDissertation}).

\begin{theorem} Let $k$ be a field of characteristic zero.
\label{GaandH1CharZeroThm}
\begin{enumerate}
\item{Every representation of $G_a$ over $k$ is of the form
$e^{xX}$ where $X$ is a nilpotent matrix over $k$, and any
nilpotent matrix $X$ gives a representation of $G_a$ according to
this formula.} \item{Every representation of $H_1$ over $k$ is of
the form $e^{xX+yY+(z-xy/2)Z}$, where $X,Y$ and $Z$ are nilpotent
matrices over $k$ satisfying $Z = [X,Y]$ and $[Z,X] = [Z,Y] = 0$,
and any such collection $X,Y,Z$ gives a representation of $H_1$
according to this formula.}
\end{enumerate}
\end{theorem}

As regards the positive characteristic theory of $G_a$, the
following is also known, due to A Suslin, E M Friedlander and C P
Bendel, 1997.

\begin{theorem} \label{GaCharpTheorem} (see proposition 1.2 of \cite{SFB})
Let $k$ be a field of positive characteristic $p$.  Then every
representation of $G_a$ over $k$ is of the form
\[ e^{X_0x}e^{X_1 x^p} \ldots e^{X_m p^m} \]
where $X_0, \ldots, X_m$ are commuting matrices with entries in
$k$ satisfying $X_i^p = 0$.  Further, any such collection of
commuting, $p$-nilpotent matrices over $k$ gives a representation
of $G_a$ according to the above formula.
\end{theorem}

Using part 1.~of theorem \ref{GaandH1CharZeroThm} and theorem
\ref{GaCharpTheorem}, we make the simple observation that, if $p
\geq \text{dimension}$, then being nilpotent and $p$-nilpotent are
identical concepts.  We see then that, for $p >>
\text{dimension}$, the characteristic $p$ representation theory of
$G_a$ and the characteristic zero theory of $G_a^\infty$ are in
perfect analogy (see chapter 11 of \cite{MyDissertation} for an
account of the representation theory of direct products).  This
has motivated the following, which is the main theorem of this
paper.

\begin{theorem} \label{TheMainTheorem}
Let $k$ be a field of characteristic $p$, and suppose $p \geq 2d$.
Then every $d$-dimensional representation of the Heisenberg group
over $k$ is of the form
\begin{equation*}
\begin{split}
 &e^{xX_0+yY_0 + (z-xy/2)Z_0} e^{x^p X_1 + y^p Y_1 + (z^p-x^p
y^p/2)Z_1}\\ &\ldots e^{x^{p^m} X_m + y^{p^m} Y_m + (z^{p^m} -
x^{p^m} y^{p^m}/2)Z_m}
\end{split}
\end{equation*}
where $X_0,Y_0,Z_0,X_1,Y_1,Z_1\ldots,X_m, Y_m,Z_m$ is a collection
of $d \times d$ nilpotent matrices over $k$ satisfying
\begin{enumerate}
\item{$[X_i,Y_i] = Z_i$ and $[Z_i,X_i] = [Z_i,Y_i] = 0$ for every
$i$} \item{whenever $i \neq j$, $X_i,Y_i$ and $Z_i$ commute with
all of $X_j,Y_j$ and $Z_j$}
\end{enumerate}
Further, any such collection of $d \times d$ matrices gives a
representation of $H_1$ over $k$ according to the above formula.
\end{theorem}

This result is perhaps more surprising than theorem
\ref{GaCharpTheorem} in that, unlike modules for $G_a$, modules
for $H_1$ in characteristic $p << \text{dim}$ generally look
hardly at all like representation for $H_1^{\infty}$ in
characteristic zero; for instance, it is not generally the case
that $[X_i,Y_i] = Z_i$ for all $i$, nor is it the case that $X_i$
commutes with $Y_j$ for $i \neq j$ (see section
\ref{counterexampleSection} for a counterexample). It is only when
$p$ becomes large enough with respect to dimension that these
relations necessarily hold.

Our method of proof is quite elementary.  We view a representation
of an algebraic group over $k$ on a vector space $V$ as a comodule
over its representing Hopf algebra (see section 3.2 of
\cite{waterhouse} or chapter 2 of \cite{hopfalgebras}), i.e. as a
$k$-linear map $\rho:V \rightarrow V \otimes A$ satisfying the
diagrams
\begin{equation}
\label{comodDiagram1}
\begin{diagram}
V & \rTo^\rho & V \otimes A \\
\dTo^\rho & & \dTo_{1 \otimes \Delta} \\
V \otimes A & \rTo_{\rho \otimes 1} & V \otimes A \otimes A \\
\end{diagram}
\end{equation}
\begin{equation}
\label{comodDiagram2}
\begin{diagram}
V & \rTo^\rho & V \otimes A \\
& \rdTo_{\isomorphic} & \dTo_{1 \otimes \varepsilon} \\
& & V \otimes k \\
\end{diagram}
\end{equation}
If we fix a basis $\{e_1, \ldots, e_n\}$ for $V$, then we can
write $\rho:e_j \mapsto \sum_i e_i \otimes a_{ij}$, where
$(a_{ij})$ is the matrix formula for the representation in this
basis. Then the diagrams above are, in equation form
\begin{equation}
\label{comodEquation1} \Delta(a_{ij}) = \sum_k a_{ik} \otimes
a_{kj}
\end{equation}
\begin{equation}
\label{comodEquation2} \varepsilon(a_{ij}) = \delta_{ij}
\end{equation}

These equations induce a combinatorial relation on certain
matrices associated to a representation, which serve as necessary
and sufficient conditions for them to define a representation.  A
systematic examination of this relation will yield the theorem.


\section{The Additive Group}

Here we record several results on the Additive group which will be
necessary to prove our main result on the Heisenberg group.  All
of what is done here was originally published in the proof of
proposition 1.2 of \cite{SFB}, but our notation differs markedly,
and so the reader may wish to consult chapter 12 of
\cite{MyDissertation} instead.

Let $k$ be any field, and let $(a_{ij})$ be a representation of
$G_a$ over $k$, which we view as an invertible matrix with entries
in $k[x]$, e.g.
\[\left(%
\begin{array}{ccc}
  1 & x & x^2 \\
  0 & 1 & 2x \\
  0 & 0 & 1 \\
\end{array}%
\right)
\]
Associated to this representation is, for each $r \in \mathbb{N}$,
the matrix of coefficients of the monomial $x^r$, which we denote
as $(c_{ij})^r$. In the above case these are given by
\[
(c_{ij})^0 = \left(%
\begin{array}{ccc}
  1 & 0 & 0 \\
  0 & 1 & 0 \\
  0 & 0 & 1 \\
\end{array}%
\right) \hspace{1cm}
(c_{ij})^1 = \left(%
\begin{array}{ccc}
  0 & 1 & 0 \\
  0 & 0 & 2 \\
  0 & 0 & 0 \\
\end{array}%
\right) \hspace{1cm}
(c_{ij})^2 = \left(%
\begin{array}{ccc}
  0 & 0 & 1 \\
  0 & 0 & 0 \\
  0 & 0 & 0 \\
\end{array}%
\right)
\]
with $(c_{ij})^r = 0$ for all other $r$.  Note that, if $(a_{ij})$
is the matrix formula for the representation, then $(a_{ij}) =
\sum_r (c_{ij})^r x^r$.

\begin{proposition} \label{GaFundamentalRelationProp}
Let $k$ be any field.  A collection of $d \times d$ matrices
$(c_{ij})^r$ define a representation of $G_a$ over $k$ if and only
if $(c_{ij})^r = 0$ for all but finitely many $r$, $(c_{ij})^0 =
1$, and that for every $r$ and $s$
\begin{equation}
\label{GaFundamentalRelation} (c_{ij})^r(c_{ij})^s = {r+s \choose
r} (c_{ij})^{r+s}
\end{equation}
\end{proposition}

\begin{proof}
See the proof of proposition 1.2 of \cite{SFB}, or section 12.1 of
\cite{MyDissertation}.

\end{proof}

\begin{corollary}
Let $k$ have characteristic zero.  Then every representation of
$G_a$ over $k$ is of the form $e^{xX}$, where $X$ is a nilpotent
matrix over $k$.
\end{corollary}

\begin{proof}
Let $(a_{ij})$ be any representation, and set $X = (c_{ij})^1$.
Using the fact that $\frac{1}{r!}$ is defined for all $r$,
examination of equation \ref{GaFundamentalRelation} yields
$(c_{ij})^r = \frac{1}{r!}X^r$. The necessity that $(c_{ij})^r$
vanish for large enough $r$ forces $X$ to be nilpotent.  Then the
matrix formula for this representation is
\begin{equation*}
\begin{split}
 (a_{ij}) &= (c_{ij})^0 + x(c_{ij})^1 + \ldots +
x^n (c_{ij})^n \\
 &= 1 + xX + \ldots + \frac{x^n X^n}{n!} \\
&= e^{xX}
\end{split}
\end{equation*}
\end{proof}

In the positive characteristic case we cannot assume that
$\frac{1}{r!}$ is defined for all $r$.  We shall need the
following.

\begin{theorem}\label{Lucas'Thm} (Lucas' theorem)
Let $n$ and $a,b, \ldots ,z$ be non-negative integers with $a + b
+ \ldots + z = n$, $p$ a prime.  Write $n = n_m p^m + n_{m-1}
p^{m-1} + \ldots + n_0$ in $p$-ary notation, similarly for $a, b,
\ldots , z$. Then, modulo $p$
\[ {n \choose a,b, \ldots, z}= \left\{
\begin{array}{c}
  0 \hspace{.3cm}\text{ \emph{if for some }$i$, $a_i + b_i + \ldots + z_i \geq p$} \\
  {n_0 \choose a_0, b_0, \ldots, z_0}{n_1 \choose a_1, b_1,
\ldots, z_1} \ldots {n_m \choose a_m, b_m, \ldots , z_m} \hspace{.3cm} \text{ \emph{otherwise}} \\
\end{array}
\right.
\]
In other words, ${n \choose a,b, \ldots, z}$ is zero if there is
some `carrying' in computing the bottom sum; otherwise, it is the
product of the multinomial coefficients of the individual digits.

\end{theorem}

\begin{corollary} Let $p$ be a prime, $n,r$ and $s$ non-negative
integers. \label{Lucas'Cor}
\begin{enumerate}
\item{${n \choose r}$ is non-zero if and only if every $p$-digit
of $n$ is greater than or equal to the corresponding $p$-digit of
$r$} \item{${r+s \choose r}$ is non-zero if and only if there is
no carrying for the sum $r+s$.}

\end{enumerate}
\end{corollary}

See \cite{LucasThm} for a proof of these facts.

\begin{proposition} \label{GaCharpLemma1}
Let $(a_{ij})$ be a representation of $G_a$ over a field $k$ of
characteristic $p>0$, given by the matricies $(c_{ij})^r$, and for
each $m$ set $X_m = (c_{ij})^{p^m}$.

\begin{enumerate}
\item{The $X_i$ commute, are nilpotent of order $\leq p$, and are
zero for all but finitely many $i$}

\item{For any $r \in \mathbb{N}$, the matrix $(c_{ij})^r$ is given
by
\[ (c_{ij})^r = \Gamma(r)^{-1} X_0^{r_0} X_1^{r_1} \ldots
X_m^{r_m} \] where $r= r_0 + r_1 p + \ldots + r_m p^m$ is the
$p$-ary expansion of $r$ and $\Gamma(r)
\stackrel{\emph{\text{def}}}{=} r_0!\ldots r_m!$}

\item{$(a_{ij})$, the matrix formula for the representation, is
given by
\[ e^{X_0x}e^{X_1x^p} \ldots e^{X_m p^m} \]}

\item{Any collection of commuting, $p$-nilpotent matrices $X_0,
\ldots, X_m$ gives a representation of $G_a$ over $k$ according to
the formula given in (3).}

\end{enumerate}

\end{proposition}

\begin{proof}
See the proof of Proposition 1.2 of \cite{SFB}, or section 12.3 of
\cite{MyDissertation}.

\end{proof}


\section{The Heisenberg Group}

For a representation of the Heisenberg group $H_1$ over a field
$k$ and $3$-tuple $(r,s,t)$ of non-negative integers, we again
define the matrix $(c_{ij})^{(r,s,t)}$ as the matrix of
coefficients of the monomial $x^r y^s z^t$.  For example, for the
representation
\[
\left(%
\begin{array}{cccccc}
  1 & 2x & x & 2x^2 & z & 2xz \\
   & 1 & 0 & x & 0 & z \\
   &  & 1 & 2x & y & 2xy \\
   &  &  & 1 & 0 & y \\
   &  &  &  & 1 & 2x \\
   &  &  &  &  & 1 \\
\end{array}%
\right)
\]
we define
\[ (c_{ij})^{(0,0,0)} = \text{Id}, \hspace{.5cm}
(c_{ij})^{(1,0,1)} =\left(
\begin{array}{cccccc}
  0 & 0 & 0 & 0 & 0 & 2 \\
   & 0 & 0 & 0 & 0 & 0 \\
   &  & 0 & 0 & 0 & 0 \\
   &  &  & 0 & 0 & 0 \\
   &  &  &  & 0 & 0 \\
   &  &  &  &  & 0 \\
\end{array}%
\right), \hspace{.5cm} (c_{ij})^{(1,1,1)} = 0
\] and so forth.  In what follows we adopt the notation, for $3$-tuples $\vec{r}$
and $\vec{s}$, $\vec{r} + \vec{s} \stackrel{\text{def}}{=}
(r_1+s_1,r_2+s_2,r_3+s_3)$.

Our first step is to work out the `fundamental relation' for
$H_1$.

\begin{proposition}
\label{H1FundamentalRelationProp} Let $k$ be any field.  A
collection $(c_{ij})^{\vec{r}}$ of matrices over $k$ defines a
representation of $H_1$ if and only if they are zero for all but
finitely many $\vec{r}$, satisfy $(c_{ij})^{(0,0,0)} =
\text{\emph{Id}}$, and for all $3$-tuples $\vec{s}$ and $\vec{t}$
\begin{equation}
\label{H1FundamentalRelation} (c_{ij})^{\vec{s}}(c_{ij})^{\vec{t}}
= \sum_{l=0}^{\text{min}(s_1,t_2)}{ s_1 + t_1 -l \choose t_1}
{s_2+t_2-l \choose s_2}{s_3 + t_3 + l \choose s_3, t_3, l}
(c_{ij})^{\vec{s} + \vec{t} + (-l,-l,l)}
\end{equation}

\end{proposition}

\begin{proof}
The first statement is immediate since the representation is
algebraic, and the second follows from equation
\ref{comodEquation2}, namely $\varepsilon(a_{ij}) = \delta_{ij}$.
For the third we examine equation \ref{comodEquation1}, namely
$\Delta(a_{ij}) = \sum_k a_{ik} \otimes a_{kj}$:
\begin{equation*}
\begin{split}
\Delta(a_{ij}) &= \Delta\left(\sum_{\vec{r}} c_{ij}^{\vec{r}}
x^{r_1} y^{r_2} z^{r_3} \right)
= \sum_{\vec{r}} c_{ij}^{\vec{r}} \Delta(x)^{r_1} \Delta(y)^{r_2} \Delta(z)^{r_3} \\
&= \sum_{\vec{r}} c_{ij}^{\vec{r}}(x \otimes 1 + 1 \otimes
x)^{r_1}(y \otimes 1 + 1 \otimes y)^{r_2} (z \otimes 1 + x \otimes
y + 1 \otimes z)^{r_3} \\
&= \sum_{\vec{r}} c_{ij}^{\vec{r}} \left[ \left(\sum_{k_1+l_1 =
r_1} {k_1 + l_1 \choose k_1} x^{k_1} \otimes x^{l_1}\right)
\left(\sum_{k_2+l_2 = r_2} {k_2
+ l_2 \choose k_2} y^{k_2} \otimes y^{l_2}\right)\right. \\
& \hspace{1.7cm} \left. \left(\sum_{k_3+l_3+m_3 = r_3}
{k_3+l_3+m_3 \choose k_3, l_3, m_3} x^{l_3} z^{k_3} \otimes
y^{l_3} z^{m_3}\right)\right] \\
&= \sum_{\vec{r}} c_{ij}^{\vec{r}} \sum_{{k_1 + l_1 = r_1} \atop
{{k_2 + l_2 = r_2} \atop {k_3 + l_3 + m_3 = r_3}}} \left[{k_1 +
l_1 \choose k_1} {k_2 + l_2 \choose k_2} {k_3+l_3+m_3 \choose
k_3,l_3,m_3} \right. \\ & \hspace{4cm}\left. x^{k_1 + l_3} y^{k_2}
z^{k_3} \otimes x^{l_1}
y^{l_2+l_3} z^{m_3}\right] \\
\end{split}
\end{equation*}
We seek to write this expression as a sum over distinct monomial
tensors, i.e.~in the form
\[ \sum_{\vec{s},\vec{t}}
\chi\left(\vec{s},\vec{t}\right) x^{s_1} y^{s_2} z^{s_3} \otimes
x^{t_1} y^{t_2} z^{t_3} \] where the summation runs over all
possible pairs of $3$-tuples and
$\chi\left(\vec{r},\vec{s}\right)$ is a scalar for each such pair.
Thus, for fixed $\vec{s}$ and $\vec{t}$ we seek non-negative
integer solutions to the system of equations
\[
\begin{array}{cc}
  k_1 + l_3 = s_1 & l_1 = t_1 \\
  k_2 = s_2 & l_2 + l_3 = t_2 \\
  k_3 = s_3 & m_3 = t_3 \\
\end{array}
\]
Once one chooses $l_3$ all other values are determined, so we
parameterize by $l_3$.  For fixed $l_3 = l$, its contribution to
the coefficient $\chi\left(\vec{s},\vec{t}\right)$ is
\[{s_1 + t_1 - l \choose t_1}{s_2+t_2-l \choose
s_2}{s_3+t_3+l \choose s_3,t_3,l}
c_{ij}^{(s_1+t_1-l,s_2+t_2-l,s_3+t_3+l)} \] and in order for such
an $l$ to induce a solution, it is necessary and sufficient that
it be no larger than either $s_1$ or $t_2$, whence we can sum the
above expression over all $0 \leq l \leq \text{min}(s_1,t_2)$ to
obtain
\[ \chi\left(\vec{s},\vec{t}\right) = \sum_{l=0}^{\text{min}(s_1,t_2)} {s_1 + t_1 - l \choose t_1}{s_2+t_2-l \choose
s_2}{s_3+t_3+l \choose s_3,t_3,l} c_{ij}^{\vec{s} + \vec{t} +
(-l,-l,l)} \] As for the right hand side of equation
\ref{comodEquation1}, one easily computes that
\[ \sum_k a_{ik} \otimes a_{kj} = \sum_{\vec{s},\vec{t}}
\left(\sum_k c_{ik}^{\vec{s}} c_{kj}^{\vec{t}} \right) x^{\vec{s}}
\otimes x^{\vec{t}} \] and upon matching coefficients for the
basis of monomial tensors we have
\[ \sum_k c_{ik}^{\vec{s}} c_{kj}^{\vec{t}} = \sum_{l=0}^{\text{min}(s_1,t_2)} {s_1 + t_1 - l \choose t_1}{s_2+t_2-l \choose
s_2}{s_3+t_3+l \choose s_3,t_3,l} c_{ij}^{\vec{s} + \vec{t} +
(-l,-l,l)} \] for every $\vec{s},\vec{t},i$ and $j$, i.e.
\[ (c_{ij})^{\vec{s}} (c_{ij})^{\vec{t}} = \sum_{l=0}^{\text{min}(s_1,t_2)} {s_1 + t_1 - l \choose t_1}{s_2+t_2-l \choose
s_2}{s_3+t_3+l \choose s_3,t_3,l} (c_{ij})^{\vec{s} + \vec{t} +
(-l,-l,l)} \] for every $\vec{s}$ and $\vec{t}$.
\end{proof}

Let $k$ have characteristic $p>0$, and $(a_{ij})$ a representation
of $H_1$ over $k$ given by the matrices $(c_{ij})^{\vec{r}}$,
$\vec{r} \in \mathbb{N}^3$. For a non-negative integer $m$, define
\begin{equation*}
\begin{split}
X_m &\stackrel{\text{def}}{=} (c_{ij})^{(p^m,0,0)} \\
Y_m &\stackrel{\text{def}}{=} (c_{ij})^{(0,p^m,0)} \\
Z_m &\stackrel{\text{def}}{=} (c_{ij})^{(0,0,p^m)} \\
\end{split}
\end{equation*}
Define also
\begin{equation*}
\begin{split}
X_{(m)} \stackrel{\text{def}}{=} (c_{ij})^{(m,0,0)} \\
Y_{(m)} \stackrel{\text{def}}{=} (c_{ij})^{(0,m,0)} \\
Z_{(m)} \stackrel{\text{def}}{=} (c_{ij})^{(0,0,m)} \\
\end{split}
\end{equation*}

Note that $H_1$ contains three copies of the additive group,
namely those matrices of the form
\[
\left(%
\begin{array}{ccc}
  1 & x & 0 \\
   & 1 & 0 \\
   &  & 1 \\
\end{array}%
\right),
\left(%
\begin{array}{ccc}
  1 & 0 & 0 \\
   & 1 & y \\
   &  & 1 \\
\end{array}%
\right) \text{ and }
\left(%
\begin{array}{ccc}
  1 & 0 & z \\
   & 1 & 0 \\
   &  & 1 \\
\end{array}%
\right)
\]
Then by propositions \ref{GaFundamentalRelationProp} and
\ref{GaCharpLemma1}, the following must hold for all $r$ and $s$:
\begin{equation}
\label{GaToH1Equation}
\begin{split}
X_{(r)} X_{(s)} &= {r + s \choose r} X_{(r+s)} \\
X_r^p &= 0 \\
X_r X_s &= X_s X_r \\
X_{(r)} &= \Gamma(r)^{-1} X_0^{r_0} \ldots X_m^{r_m} \\
\end{split}
\end{equation}
Identical statements hold if we replace $X$ with $Y$ or $Z$.

As mentioned in the introduction, there is no reason to suspect
that modules for $H_1$ in characteristic $p>0$ bear any
resemblance to modules for $H_1^\infty$ in characteristic zero,
unless $p$ is large enough with respect to dimension.

\begin{lemma}
\label{H1remark} Suppose that $p$ is greater than or equal to
twice the dimension of a representation, and that the sum $r + s$
carries. Then at least one of $P_{(r)}$ or $Q_{(s)}$ must be zero,
where $P$ and $Q$ can be any of $X$, $Y$ or $Z$.
\end{lemma}

\begin{proof}
Since the $X_i$, $Y_i$, and $Z_i$ are all nilpotent, they are
nilpotent of order less than or equal to the the dimension of the
representation, which we assume is no greater than $p/2$. Since
the sum $r + s$ carries, we have $r_i + s_i \geq p$ for some $i$,
whence, say, $r_i \geq p/2$.  Then
\[ P_{(r)} = \Gamma(r)^{-1}P_0^{r_0} \ldots P_i^{r_i} \ldots
P_i^{r_m} \] is zero, since $P_i^{r_i}$ is.

\end{proof}

\begin{proposition}
\label{H1CharpLargeThm1} Suppose $p$ is greater than or equal to
twice the dimension of a representation. Then the following
relations hold:

\begin{enumerate}
\item{$[Z_n,X_m] = [Z_n,Y_m] = 0$ for every $n$ and $m$}

\item{$[X_m,Y_m] = Z_m$ for every $m$}

\item{$[X_n,Y_m] = 0$ for every $n \neq m$}
\end{enumerate}

\end{proposition}

\begin{proof}
To prove (2), consider equation \ref{H1FundamentalRelation}
applied to $X_m Y_m$:
\begin{equation*}
\begin{split}
X_m Y_m &=  (c_{ij})^{(p^m,0,0)} (c_{ij})^{(0,p^m,0)} \\
 &= Y_m X_m +
\left(\sum_{l=1}^{p^m-1} {p^m -l \choose 0}{p^m - l \choose 0} {l
\choose l} Z_{(l)} Y_{(p^m-l)} X_{(p^m-l)} \right) + Z_m \\
\end{split}
\end{equation*}
For every $0 < l < p^m$ there is clearly some carrying in
computing the sum $(p^m -l) + l$, so by lemma \ref{H1remark} the
summation term $Z_{(l)}Y_{(p^m-l)}X_{(p^m-l)}$ is always zero,
since at least one of $Z_{(l)}$ or $Y_{(p^m-l)}$ is zero. This
gives $Z_m = [X_m,Y_m]$ as claimed.

To prove (3), consider equation \ref{H1FundamentalRelation}
applied to $X_m Y_n$ for $m \neq n$:

\begin{equation*}
\begin{split}
X_m Y_n &= (c_{ij})^{(p^m,0,0)}(c_{ij})^{(0,p^n,0)} \\
&= Y_n X_m + \left(\sum_{l=1}^{\text{min}(p^n,p^m)} {p^m -l
\choose 0}{p^n - l \choose 0} {l \choose l} Z_{(l)} Y_{(p^n-l)}
X_{(p^m-l)} \right)
\end{split}
\end{equation*}
In case $m < n$, for every value of $l$ in the above summation,
the sum $(p^n -l) + l$ carries, forcing at least one of $Z_{(l)}$
or $Y_{(p^n-l)}$ to be zero, forcing every term in the summation
to be zero. A similar statement holds in case $n < m$. This proves
$X_m Y_n = Y_n X_m$, as claimed.

(1) is in fact true without any hypothesis on the characteristic.
To prove it, apply equation \ref{H1FundamentalRelation} to $X_n
Z_m$ and $Z_m X_n$, for which you get the same answer, and the
same can be done to show $[Y_m,Z_n] = 0$.

\end{proof}

We have shown thus far that, if $p \geq 2d$, every $d$-dimensional
representation in characteristic $p$ is given by a finite sequence
$X_i,Y_i,Z_i$ of $d \times d$ matrices over $k$ satisfying

\begin{enumerate}
\item{The $X_i$, $Y_i$, and $Z_i$ are all nilpotent} \item{$Z_i
=[X_i,Y_i]$ for every $i$} \item{$[X_i,Z_i] = [Y_i,Z_i] = 0$ for
every $i$} \item{For every $i \neq j$, $X_i,Y_i,Z_i$ all commute
with $X_j,Y_j,Z_j$}
\end{enumerate}

We now show sufficiency of these relations.

\begin{lemma}
\label{H1MainThmLemma} Let $k$ have characteristic $p>0$, and let
$X,Y$ and $Z$ be $p$-nilpotent matrices over $k$ satisfying
$Z=[X,Y]$ and $[Z,X] = [Z,Y] = 0$.  Then for $0 \leq m,n < p$,
\[X^nY^m = \sum_{l=0}^{\text{min}(n,m)}l!{n \choose l}{m \choose l}Z^lY^{m-l} X^{n-l}\]
\end{lemma}

Remark: This result is in fact true when $\text{char}(k) = 0$,
which is the key fact which can be used to prove part (2) of
theorem \ref{GaandH1CharZeroThm}.

\begin{proof}
We proceed by a double induction on $n$ and $m$.  If $n$ or $m$ is
zero the result is trivial, and if $n=m=1$ the equation is $XY =
YX + Z$, which is true by assumption.  Consider then $X^nY$, and
by induction suppose that $X^{n-1}Y = YX^{n-1} + (n-1)ZX^{n-2}$.
Then using the relation $XY = Z+YX$ and $X$ commuting with $Z$ we
have
\begin{equation*}
\begin{split}
 X^nY &= X^{n-1}XY \\ &
 = X^{n-1}(Z+YX) \\
 &= ZX^{n-1} + (X^{n-1}Y)X \\
 &= ZX^{n-1} + (YX^{n-1} + (n-1)ZX^{n-2})X \\
 &= nZX^{n-1} + YX^n
\end{split}
\end{equation*}
 and
so the equation is true when $m=1$.  Now suppose that $m \leq n$,
so that $\text{min}(n,m) = m$.  Then
\begin{equation*}
\begin{split}
 X^n Y^m &= (X^nY)Y^{m-1} \\
 &= (YX^n + nZX^{n-1})Y^{m-1} \\
 &= Y(X^nY^{m-1}) + nZ(X^{n-1}Y^{m-1})
\end{split}
\end{equation*}
which by induction is equal to
\begin{equation*}
\begin{split}
&= Y \left(\sum_{l=0}^{m-1} l!{n \choose l}{m-1 \choose
l}Z^lY^{m-1-l}X^{n-l} \right) \\
& \qquad+ n Z \left(\sum_{l=0}^{m-1}l!{n-1
\choose l}{m-1 \choose l}Z^lY^{m-1-l}X^{n-1-l} \right) \\
&= \sum_{l=0}^{m-1} l! {n \choose l}{m-1 \choose
l}Z^lY^{m-l}X^{n-l} \\
& \qquad+ \sum_{l=0}^{m-1} n l!{n-1 \choose l}{m-1
\choose l} Z^{l+1}Y^{m-1-l}X^{n-1-l} \\
&= Y^mX^n + \sum_{l=1}^{m-1} l!{n \choose l}{m-1 \choose l}Z^l
Y^{m-l}X^{n-l} \\
& \qquad + \sum_{l=1}^m n(l-1)!{n-1 \choose l-1}{m-1 \choose
l-1}Z^l Y^{m-l}X^{n-l}
\end{split}
\end{equation*}
where, in the last step, we have chopped off the first term of the
first summation and shifted the index $l$ of the second summation.
If we chop off the last term of the second summation we obtain
\begin{equation*}
\begin{split}
 = Y^mX^n &+ \sum_{l=1}^{m-1} l!{n \choose l}{m-1 \choose l}Z^l
Y^{m-l}X^{n-l} \\
&+ \sum_{l=1}^{m-1} n(l-1)!{n-1 \choose l-1}{m-1 \choose l-1}Z^l
Y^{m-l}X^{n-l} \\
&+ n(m-1)!{n-1 \choose m-1}{m-1 \choose m-1}Z^m X^{n-m}
\end{split}
\end{equation*}
and upon merging the summations, we have
\begin{equation*}
\begin{split}
 &= Y^m X^n + \sum_{l=1}^{m-1} \left[l!{n \choose l}{m-1 \choose
l} + n(l-1)!{n-1 \choose l-1}{m-1 \choose l-1} \right] Z^l Y^{m-l}
X^{n-l} \\
& \qquad + n(m-1)!{n-1 \choose m-1}{m-1 \choose m-1} Z^mX^{n-m} \\
&= Y^m X^n + \sum_{l=1}^{m-1} \left[l!{n \choose l}{m-1 \choose l}
+ n(l-1)!{n-1 \choose l-1}{m-1 \choose l-1} \right] Z^l Y^{m-1}
X^{n-l} \\
& \qquad + m!{n \choose m}{m \choose m} Z^mX^{n-m} \\
&= Y^m X^n + \sum_{l=1}^{m-1}\left[l!{n \choose l}{m \choose l}
\right] Z^l Y^{m-l} X^{n-l} \\
& \qquad + n!{n \choose m}{m \choose m} Z^m X^{n-m} \\
&= \sum_{l=0}^m l!{n \choose l}{m \choose l} Z^l Y^{m-l} X^{n-l}
\end{split}
\end{equation*}
This proves the case of $m \leq n$, and the case of $n \leq m$ is
hardly any different, and left to the reader.

\end{proof}

\begin{theorem}
\label{H1CharpLargeThm2} Suppose $p \geq 2d$.  Let $X_i$, $Y_i$
and $Z_i$ be a finite sequence of $d \times d$ matrices over a
field $k$ of characteristic $p$ satisfying

\begin{enumerate}
\item{The $X_i$, $Y_i$, and $Z_i$ are all nilpotent} \item{$Z_i
=[X_i,Y_i]$ for every $i$} \item{$[X_i,Z_i] = [Y_i,Z_i] = 0$ for
every $i$} \item{For every $i \neq j$, $X_i,Y_i,Z_i$ all commute
with $X_j,Y_j,Z_j$}

\end{enumerate}
Let $n = n_m p^m + n_{m-1} p^{m-1} + \ldots + n_1p+n_0$, and
assign
\[ X_{(n)} = \Gamma(n)^{-1} X_0^{n_0} \ldots X_m^{n_m} \]
and similarly for $Y_{(n)}$ and $Z_{(n)}$.  Set
\[ (c_{ij})^{(n,m,k)} = Z_{(k)}Y_{(m)}X_{(n)} \]
Then these assignments define a valid $d$-dimensional
representation of $H_1$ over $k$.

\end{theorem}

\begin{proof}
The first two conditions of proposition
\ref{H1FundamentalRelationProp} are immediate.  Then for arbitrary
$n,m,k,r,s,t \in \mathbb{N}$, the equation we must verify is
\begin{gather*}
(c_{ij})^{(n,m,k)}(c_{ij})^{(r,s,t)} \\
= \sum_{l=0}^{\text{min}(n,s)} {n+r-l \choose r}{m+s-l \choose
m}{k+t+l \choose k,t,l} (c_{ij})^{(n+r-l,m+s-l,k+t+l)}
\end{gather*}
 which, with
the given assignments and assumptions, can be written
\begin{gather*}
 Z_{(k)}Z_{(t)}Y_{(m)}X_{(n)} Y_{(s)} X_{(r)} \\
 = \sum_{l=0}^{\text{min}(n,s)} {n+r-l \choose r}{m+s-l \choose
m}{k+t+l \choose k,t,l} Z_{(k+t+l)} Y_{(m+s-l)} X_{(n+r-l)}
\end{gather*}
Equation \ref{GaToH1Equation} gives the identities
\begin{equation*}
\begin{split}
 Z_{(k)} Z_{(t)} &= {k+t \choose t} Z_{(k+t)} \\
 Y_{(m)}Y_{(s-l)}  &= {m+s-l \choose m}Y_{(m+s-l)} \\
X_{(n-l)} X_{(r)}  &= {n+r-l \choose r} X_{(n+r-l)}
\end{split}
\end{equation*}
 so we can
rewrite our equation as
\[ {k+t \choose t} Z_{(k+t)}Y_{(m)}X_{(n)} Y_{(s)} X_{(r)} = \sum_{l=0}^{\text{min}(n,s)} {k+t+l \choose k,t,l} Z_{(k+t+l)}
Y_{(m)}Y_{(s-l)} X_{(n-l)} X_{(r)} \] First suppose that the sum
$k+t$ carries. In this case the equation is true, since the left
hand side
 binomial coefficient vanishes, and the right hand side multinomial coefficient
 vanishes for every $l$, causing both sides to be zero.  We assume
 then that $k+t$ does not carry, so we can divide both sides
 by ${k+t \choose t}$ to yield
 \[ Z_{(k+t)}Y_{(m)}X_{(n)} Y_{(s)} X_{(r)} = \sum_{l=0}^{\text{min}(n,s)} {k+t+l \choose l} Z_{(k+t+l)}
Y_{(m)}Y_{(s-l)} X_{(n-l)} X_{(r)} \] Now apply ${k+t+l \choose
l}Z_{(k+t+l)} = Z_{(k+t)} Z_{(l)}$:
 \[ Z_{(k+t)}Y_{(m)}X_{(n)} Y_{(s)} X_{(r)} = \sum_{l=0}^{\text{min}(n,s)} Z_{(k+t)} Z_{(l)} Y_{(m)}Y_{(s-l)}
X_{(n-l)} X_{(r)} \] We have $Z_{(k+t)}$ in the front and
$X_{(r)}$ in the rear of both sides, so it suffices to show
 \[ Y_{(m)}X_{(n)} Y_{(s)} = \sum_{l=0}^{\text{min}(n,s)} Z_{(l)} Y_{(m)}Y_{(s-l)} X_{(n-l)}
\]
and since $Y_{(m)}$ commutes with $Z_{(l)}$, we can move it to the
front of the right hand side, and then take it off both sides, so
it suffices to show
\begin{equation}
\label{H1charpProofEq} X_{(n)} Y_{(m)} =
\sum_{l=0}^{\text{min}(n,m)} Z_{(l)} Y_{(m-l)} X_{(n-l)}
\end{equation}
where we have replaced $s$ with the more traditional $m$.

Now we begin to replace the $X_{(i)}'s$ with their definitions in
terms of the $X_i's$, similarly for $Y$ and $Z$, so that the left
hand side of equation \ref{H1charpProofEq} is
\[ \left[\Gamma(n) \Gamma(m)\right]^{-1}X_0^{n_0} \ldots
X_k^{n_k}Y_0^{m_0} \ldots Y_k^{m_k} \] and since everything
commutes except $X_i$ and $Y_j$ when $i = j$, we can write
\[ \left[\Gamma(n) \Gamma(m)\right]^{-1}(X_0^{n_0}Y_0^{m_0}) \ldots
(X_k^{n_k}Y_k^{m_k}) \] Moving all coefficients to the right, we
must show
\begin{equation}
\label{lastEquation} (X_0^{n_0}Y_0^{m_0}) \ldots
(X_k^{n_k}Y_k^{m_k}) = \Gamma(n) \Gamma(m)
\sum_{l=0}^{\text{min}(n,m)} Z_{(l)} Y_{(m-l)} X_{(n-l)}
\end{equation}
We proceed by induction on $k$, the maximum number of $p$-digits
of either $m$ or $n$. If $k=0$ the equation is
\begin{equation*}
\begin{split}
X_0^{n_0}Y_0^{m_0} &= n_0! m_0! \sum_{l=0}^{\text{min}(n_0,m_0)}
Z_{(l)} Y_{(m_0-l)} X_{(n_0 -l)} \\
 &= \sum_{l=0}^{\text{min}(n_0,m_0)}
\frac{n_0!m_0!}{(m_0-l)!(n_0-l)!l!} Z_0^l Y_0^{m_0-l} X_0^{n_0-l}
\\
 &= \sum_{l=0}^{\text{min}(n_0,m_0)} l!{ n_0 \choose l}{ m_0 \choose
l}  Z_0^l Y_0^{m_0-l} X_0^{n_0-l}
\end{split}
\end{equation*}
which is true by lemma \ref{H1MainThmLemma} applied to $X_0,Y_0$
and $Z_0$. Now suppose the equation is true when $n$ and $m$ have
no more than $k-1$ digits. Let $n=n_{k-1}p^{k-1} + \ldots + n_0$
and let $n^\prime = n_{k}p^{k} + n_{k-1} p^{k-1} + \ldots + n_0$,
and similarly for $m$. Then by induction we have
\begin{equation*}
\begin{split}
(X_0^{n_0}Y_0^{m_0}) \ldots (X_k^{n_k}Y_k^{m_k}) &=
\left[(X_0^{n_0}Y_0^{m_0}) \ldots
(X_{k-1}^{n_{k-1}}Y_{k-1}^{m_{k-1}})\right](X_{k}^{n_{k}}Y_{k}^{m_{k}})
\\
 &= \left(\Gamma(n) \Gamma(m) \sum_{l=0}^{\text{min}(n,m)} Z_{(l)}
Y_{(m-l)} X_{(n-l)}\right) \\
& \qquad \left( \sum_{l^\prime=0}^{\text{min}(n_{k},m_{k})}
l^\prime!{ n_{k} \choose l^\prime}{ m_{k} \choose l^\prime}
Z_{k}^{l^\prime}
Y_{k}^{m_{k}-l^\prime} X_{k}^{n_{k}-l^\prime} \right) \\
 &= n_k!\Gamma(n)m_k!\Gamma(m) \sum_{l,l^\prime} \left[ \left(\frac{Z_{(l)}
Z_k^{l^\prime}}{l^\prime!} \right) \left(\frac{Y_{(m-l)}
Y_k^{m_k-l^\prime}}{(m_k-l^\prime)!} \right) \right.\\
& \qquad \left. \left(\frac{X_{(n-l)}
X_k^{n_k-l^\prime}}{(n_k-l^\prime)!} \right) \right]
\end{split}
\end{equation*}
Note that these divisions are valid, since for every value of
$l^\prime$ in the summation, $l^\prime \leq m_k,n_k < p$.  Note
also that, since $l \leq p^{k-1}$ and $l^\prime < p$ for all
values of $l,l^\prime$ in the summation, theorem \ref{Lucas'Thm}
gives that ${ l + l^\prime p^k \choose l} = 1$ for all such $l$
and $l^\prime$. For similar reasons we have
${(m-l)+(m_k-l^\prime)p^k \choose m-l} = {(n-l)+(n_k -
l^\prime)p^k \choose n-l} = 1$.  Then we have the identities
\begin{equation*}
\begin{split}
 n_k!\Gamma(n) &= \Gamma(n^\prime) \\
m_k!\Gamma(m)  &= \Gamma(m^\prime) \\
 \frac{Z_{(l)} Z_k^{l^\prime}}{l^\prime!} &= Z_{(l)}Z_{(l^\prime
p^k)} \\
&= {l+l^\prime p^k \choose l} Z_{(l+l^\prime p^k)} \\
&= Z_{(l+l^\prime p^k)} \\
 \frac{Y_{(m-l)} Y_k^{m_k-l^\prime}}{(m_k-l^\prime)!} &=
Y_{(m-l)}Y_{((m_k-l^\prime)p^k)}\\
&= {(m-l)+(m_k-l^\prime)p^k \choose m-l}
Y_{((m+m_kp^k)-(l+l^\prime)p^k)} \\
&= Y_{(m^\prime-(l+l^\prime p^k))}
\end{split}
\end{equation*}
 and similarly
\[ \frac{X_{(n-l)} X_k^{n_k-l^\prime}}{(n_k-l^\prime)!} =
X_{(n^\prime-(l+l^\prime p^k))} \] These substitutions transform
the right hand side of our equation into
\[ = \Gamma(n^\prime) \Gamma(m^\prime) \sum_{l,l^\prime}
Z_{(l+l^\prime p^k)} Y_{(m^\prime-(l+l^\prime p^k))}
X_{(n^\prime-(l+l^\prime p^k))} \] But, if we look at the
summation limits of $l = 0 \ldots \text{min}(n,m)$ and $l^\prime =
0 \ldots \text{min}(n_k,m_k)$, we see that it is really a single
summation running from $0$ to $\text{min}(n^\prime,m^\prime)$,
with $l+l^\prime p^k$ as the summation variable. That is
\[ = \Gamma(n^\prime) \Gamma(m^\prime)
\sum_{l=0}^{\text{min}(n^\prime,m^\prime)} Z_{(l)}
Y_{(m^\prime-l)} X_{(n^\prime-l)} \] which proves equation
\ref{lastEquation}.  This completes the proof.
\end{proof}

The main theorem of this paper, theorem \ref{TheMainTheorem},  is
now proved, modulo the fact that, for $\text{char}(k) \geq 2 d$,
such a collection $X_i,Y_i,Z_i$ of $d \times d$ matrices generate
a representation according to the formula
\begin{equation*}
\begin{split}
 &e^{xX_0+yY_0 + (z-xy/2)Z_0} e^{x^p X_1 + y^p Y_1 + (z^p-x^p
y^p/2)Z_1}\\ &\ldots e^{x^{p^m} X_m + y^{p^m} Y_m + (z^{p^m} -
x^{p^m} y^{p^m}/2)Z_m}
\end{split}
\end{equation*}
The proof of this is but a slight generalization of what has
already been done in characteristic zero.  We simply take the
standard Lie group proof of the Baker-Campbell-Hausdorff formula
for the Heisenberg group (see theorem 13.1 of \cite{hall}),
replace `derivative' with `formal derivative' of polynomials, and
show that the necessary results hold in characteristic $p$ when
$p$ is sufficiently large. We direct the interested reader to
section 13.5 of \cite{MyDissertation} for the full proof.


\section{Counterexamples and Sharpness}
\label{counterexampleSection}

If $\text{char}(k) = p$ is not large enough with respect to
dimension, conditions (1) and (2) of theorem \ref{TheMainTheorem}
do not necessarily hold.  Here we give two such examples.

Let $k$ be the finite field $\mathbb{Z}_2$, and let $V$ be the
$10$-dimensional sub-coalgebra of the Hopf algebra $A = k[x,y,z]$
given by
\[ V = \text{span}_k(1,x,y,z,x^2,xy,xz,y^2,yz,z^2) \]
i.e.~the span of all monomials of degree no greater than $2$,
endowed with the structure of a right $A$-comodule $\rho:V
\rightarrow V \otimes A$ given by the restriction of $\Delta$ to
$V$.  The corresponding representation, in this ordered basis, has
matrix formula
\[M =
\left(%
\begin{array}{ccccccccccc}
  1 & x & y & z & x^2 & xy & xz & y^2 & yz & z^2 \\
   & 1 & 0 & y & 0 & y & z+xy & 0 & y^2 & 0 \\
   & & 1 & 0 & 0 & x & 0 & 0 & z & 0 \\
   &  & & 1 & 0 & 0 & x & 0 & y & 0 \\
  &  &  &  &  1 & 0 & y & 0 & 0 & y^2 \\
    &  &  &  &  &1 & 0 & 0 & y & 0 \\
    &  &  &  &  &  & 1 & 0 & 0 & 0 \\
   &  &  &  &  &  &  & 1 & 0 & 0 \\
  &  &  &  &  &  &  &  & 1 & 0 \\
  &  &  &  &  &  &  &  &  & 1 \\
\end{array}%
\right)
\]
Let $X_0$ denote the matrix of coefficients of the monomial $x$ in
$M$, i.e.~the $10 \times 10$ matrix with $1's$ in its
$(1,2),(3,6)$, and $(4,7)$ entries, and $0's$ elsewhere. Similarly
define $Y_1$ as the matrix of coefficients of the monomial $y^2$.
Then one can check by hand that $X_0Y_1 - Y_1 X_0 \neq 0$, which
contradicts condition (2) of theorem \ref{TheMainTheorem}.

For a counterexample to condition (1) of theorem
\ref{TheMainTheorem}, again let $k = \mathbb{Z}_2$, and consider
the span of all monomials of degree no greater than $3$
\[ V = \text{span}_k(1, x, y, z, x^2, xy, xz, y^2, yz,
z^2, x^3, x^2y, x^2z, xy^2, xyz, xz^2, y^3, y^2z, yz^2, z^3)\]
This is a $20$-dimensional representation, who's matrix formula
would not fit on this page, but the intrepid reader can verify by
hand that $[X_1,Y_1] \neq Z_1$.

One may ask: is the condition $ p \leq 2 \text{dim}$ sharp?  That
is, given any pair $(p,d)$ with $p < 2d$, does there exists a
$d$-dimensional representation of $H_1$ over a field of
characteristic $p$ such that at least one of conditions (1) and
(2) of theorem \ref{TheMainTheorem} do not hold?  The answer is
no, at least for small dimensions.  For example, we claim that any
$2$-dimensional representation over \emph{any} positive
characteristic field necessarily satisfies these conditions; in
particular, even when $\text{char}(k) = 2$.  As any representation
of a unipotent algebraic group is upper-triangular, and since all
of $X_i, Y_i$ and $Z_i$ are nilpotent, we can take them all to be
scalar multiples of
\[
\left(%
\begin{array}{cc}
  0 & 1 \\
  0 & 0 \\
\end{array}%
\right)
\]
whence $X_i,Y_j$ and $Z_k$ obviously all commute for all $i,j$ and
$k$. We claim further that $Z_m = 0$ for all $m$.  To see this,
consider
\begin{equation*}
\begin{split}
X_m Y_m &=  (c_{ij})^{(p^m,0,0)} (c_{ij})^{(0,p^m,0)} \\
 &= Y_m X_m +
\left(\sum_{l=1}^{p^m-1} {p^m -l \choose 0}{p^m - l \choose 0} {l
\choose l} Z_{(l)} Y_{(p^m-l)} X_{(p^m-l)} \right) + Z_m \\
\end{split}
\end{equation*}
(see the proof of proposition \ref{H1CharpLargeThm1}) which, since
$X_m Y_m = Y_m X_m$, can be written
\[ 0 = Z_m + \left(\sum_{l=1}^{p^m-1} Z_{(l)} Y_{(p^m-l)} X_{(p^m-l)} \right) \]
which, in view of equation \ref{GaToH1Equation} can be written
\[ 0 = Z_m +  \left(\sum_{l=1}^{p^m-1} \Gamma(l)^{-1} \Gamma(p^m-l)^{-2} Z_{0}^{l_0} \ldots Z_{m-1}^{l_{m-1}}
Y_0^{r_0} \ldots Y_{m-1}^{r_{m-1}} X_0^{r_0} \ldots
X_{m-1}^{r_{m-1}} \right)
\]
where the $r_i$ are the $p$-digits of $p^m  - l$.  Let $0 < l <
p^m$.  Then the sum $(p^m - l) + l$ carries, whence, for some $i$,
$l_i$ and $r_i$ are both non-zero.  Then the summation term
\[ Z_{0}^{l_0} \ldots Z_{m-1}^{l_{m-1}}
Y_0^{r_0} \ldots Y_{m-1}^{r_{m-1}} X_0^{r_0} \ldots
X_{m-1}^{r_{m-1}} \] vanishes, since this entire product commutes
and $Z_i^{l_i} Y_i^{r_i} = 0$.  Thus the entire summation is zero,
whence so is $Z_m$.  This shows that every positivie
characteristic $2$-dimensional representation of $H_1$ satisfies
conditions (1) and (2) of theorem \ref{TheMainTheorem}, even when
$\text{char}(k) = 2 < 2 \text{ dimension } = 4$.


\section{Further Directions}

The last paragraph of the last section shows that the condition $p
\geq 2d$ is not sharp.  However, we suspect that this condition is
\emph{assymptotically} sharp, in one of the following senses.

\begin{conjecture}
There exists a prime $q$ such that, for every prime $p \geq q$,
there exists a $\frac{p+1}{2}$-dimensional module for $H_1$ over a
field of characteristic $p$ which does satisfy not at least one of
conditions (1) and (2) of theorem \ref{TheMainTheorem}.
\end{conjecture}

Or, perhaps the weaker

\begin{conjecture}
For arbitrarily large primes $p$, there exists a
$\frac{p+1}{2}$-dimensional module for $H_1$ over a field of
characteristic $p$ which does not satisfy at least one of
conditions (1) and (2) of theorem \ref{TheMainTheorem}.
\end{conjecture}

But more importantly, the author strongly suspects that results
analogous to theorems \ref{GaCharpTheorem} and
\ref{TheMainTheorem} for the Additive and Heisenberg groups should
apply to a much wider class of unipotent algebraic groups. The
author has in fact proved this result for all of the so-called
generalized Heisenberg groups (though this is not in print), and
at the very least believes that a proof for all of the unipotent
upper triangular groups, the most important class of unipotent
groups, will soon appear in a sequel.

Perhaps a more interesting question, and one which has not been at
all addressed in this paper, is not \emph{that} this phenomenon
might be true for unipotent groups in general, but rather
\emph{why} it might be true of unipotent groups.  An algebraic
group is unipotent if and only if, over a field of characteristic
$p>0$, it is of exponent a power of $p$. The author believes that
it is exactly this property of unipotent groups which makes this
theorem true, along with perhaps a clever appeal to the
compactness theorem for first-order logic (but in what language?).
The author again hopes that such insights will be forthcoming in a
sequel.


\section*{Acknowledgements}

The author would like to sincerely thank his thesis advisor, Paul
Hewitt, under who's direction and advice these results originally
appeared. Thanks also to Dave Hemmer, for his thoughtful advice
and encouragement, and to Chris Bendel, for his generous reading
of this manuscript and invaluable suggestions.

\bibliography{common}

\begin{thebibliography}{1}

\bibitem{MyDissertation}
Michael Crumley.
\newblock {\em Ultraproducts of Tannakian Categories and Generic Representation
  Theory of Unipotent Algebraic Groups}.
\newblock PhD thesis, The University of Toledo, Department of Mathematics,
  2010.

\bibitem{hopfalgebras}
Nastasescu Dascalescu and Raianu.
\newblock {\em Hopf Algebras: An Introduction}.
\newblock Pure and Applied Mathematics. Marcel Dekker, New York, 2001.

\bibitem{SFB}
A~Suslin E~M~Friedlander and C~P Bendel.
\newblock Infinitesimal 1-parameter subgroups and cohomology.
\newblock {\em Journal of the AMS}, 10(3):693--728, July 1997.

\bibitem{LucasThm}
N.J. Fine.
\newblock Binomial coefficients modulo a prime.
\newblock {\em Amer. Math. Monthly}, pages 589--592, Dec 1974.

\bibitem{hall}
Brian~C. Hall.
\newblock {\em Lie Groups, Lie Algebras, and Representations: An Elementary
  Introduction}.
\newblock Graduate Texts in Mathematics. Springer-Verlag, New York, 2003.

\bibitem{waterhouse}
William~C. Waterhouse.
\newblock {\em Introduction to Affine Group Schemes}.
\newblock Graduate Texts in Mathematics. Springer-Verlag, New York, 1979.

\end{thebibliography}
\bibliographystyle{plain}

\end{document}